\documentclass[12pt]{amsart}
\usepackage{amssymb,latexsym}
\usepackage{enumerate}

\makeatletter
\@namedef{subjclassname@2010}{%
  \textup{2010} Mathematics Subject Classification}
\makeatother

\newtheorem{thm}{Theorem}[section]
\newtheorem{cor}[thm]{Corollary}
\newtheorem{lem}[thm]{Lemma}

\theoremstyle{definition}

\newtheorem{rem}[thm]{Remark}
\newtheorem{exa}[thm]{Example}

\numberwithin{equation}{section}

\newcommand{\R}{\mathbb{R}}
\newcommand{\C}{\mathbb{C}}
\newcommand{\X}{\mathcal{X}}
\newcommand{\XM}{\X(M)}
\newcommand{\XU}{\X(U)}

\newcommand{\const}{\operatorname{const}}

\newcommand{\D}{\operatorname{D}}

\newcommand{\DD}{\mathcal{D}}

\newcommand{\Span}{\operatorname{span}}

\newcommand{\FiXiEta}{{(\varphi,\xi,\eta)}}
\newcommand{\FI}{{\varphi}}

\newcommand{\rank}{\operatorname{rank}}

\newcommand{\Jt}{\widetilde{J}}
\newcommand{\Ct}{\widetilde{\mathbb{C}}}

\begin{document}


\baselineskip=17pt



\title[Parallel Almost Paracontact Structures...]{Parallel Almost Paracontact Structures on Affine Hypersurfaces}

\author[Z. Szancer]{Zuzanna Szancer}
\address{Department of Applied Mathematics \\ University of Agriculture in Krakow\\ 253c Balicka Street\\
30-198 Krakow, Poland}
\email{Zuzanna.Szancer@urk.edu.pl}

\date{}

\begin{abstract}
Let $\Jt$ be the canonical para-complex structure on $\R^{2n+2}\simeq\widetilde{\mathbb{C}}^{n+1}$. We study real affine hypersurfaces $f\colon M\rightarrow \widetilde{\mathbb{C}}^{n+1}$ with a $\Jt$-tangent transversal vector field. Such vector field induces in a natural way an almost paracontact structure $\FiXiEta$ on $M$ as well as the affine connection $\nabla$. In this paper we give the classification of hypersurfaces with the property that
$\varphi$ or $\eta$ is parallel relative to the connection $\nabla$. Moreover, we show that if $\nabla\varphi=0$
(respectively $\nabla\eta=0$) then around each point of $M$ there exists a parallel almost paracontact structure. Results we illustrate with some examples.
\end{abstract}



\maketitle

\section{Introduction}
\par Para-complex and paracontact structures were studied by many authors over past decades. These structures play an important role in pseudo-Riemannian
geometry as well as modern mathematical physics. In particular, some recent results related to paracontact geometry can be found in \cite{Z,CKM,KM}.
Moreover, recently some relations between para-complex and affine differential geometry were also studied (see \cite{LS,CLS} and \cite{SZ}).
\par If we denote by $\Jt$ the canonical para-complex structure on $\R^{2n+2}\simeq\widetilde{\mathbb{C}}^{n+1}$ then, in a similar way like in the
complex case (\cite{Cruceanu,SzanSzan,SZ2}), one may consider affine hypersurfaces $f\colon M\rightarrow \widetilde{\mathbb{C}}^{n+1}$ with a $\Jt$-tangent transversal vector field. Some recent results for affine hypersurfaces with a $\Jt$-tangent transversal vector field
can be found in \cite{S1,S3}.
\par In \cite{S4} the author studied real affine hypersurfaces of the complex space $\C^{n+1}$ with a $J$-tangent transversal vector field and an induced almost contact structure $\FiXiEta$ such that $\varphi$ or $\eta$ is parallel relative to the induced
affine connection. Now, it is natural to ask what happens in a para-complex situation. It is worth highlighting that in this case we do not have the canonical
$\Jt$-tangent transversal vector field (the Riemannian normal field in general is not $\Jt$-tangent), so the situation is more complex.
\par In Section 2 we briefly recall basic formulas of affine differential geometry.
\par In Section 3 we recall the definition of an almost paracontact structure introduced for the first time in \cite{KW}. We recall the notion of a $\Jt$-tangent transversal vector field and a $\Jt$-invariant distribution. We also recall some results obtained in \cite{S1,S3} for induced almost paracontact structures which we will use in the next section.
\par Section 4 contains main results of this paper. First we show some basic relations which hold among induced objects under an additional condition
 that either $\varphi$ or $\eta$ is $\nabla$-parallel.
 Later we show that if any of the above is satisfied for some $\Jt$-tangent transversal vector field then we can
 always find (at least locally) another $\Jt$-tangent transversal vector field such that the induced almost paracontact structure is parallel.
 Finally, we provide the full local classification of the above mentioned hypersurfaces. In order to illustrate the results some examples are also given.
 In particular we show that (contrary to the case when $\nabla\eta=0$ or $\nabla\varphi=0$) the condition $\nabla\xi=0$ is much more weaker.
\section{Preliminaries}
We briefly recall the basic formulas of affine differential
geometry. For more details, we refer to \cite{NS}. Let $f\colon M\rightarrow\R^{n+1}$ be an orientable,
connected differentiable $n$-dimensional hypersurface immersed in
affine space $\R^{n+1}$ equipped with its usual flat connection
$\D$. Then, for any transversal vector field $C$ we have
\begin{equation}\label{eq::FormulaGaussa}
\D_Xf_\ast Y=f_\ast(\nabla_XY)+h(X,Y)C
\end{equation}
and
\begin{equation}\label{eq::FormulaWeingartena}
\D_XC=-f_\ast(SX)+\tau(X)C,
\end{equation}
where $X,Y$ are vector fields tangent to $M$. The formulas (\ref{eq::FormulaGaussa}) and (\ref{eq::FormulaWeingartena}) are called the formula of Gauss and the formula of Weingarten, respectively. For any transversal vector field $\nabla$ is a torsion-free connection, $h$ is a symmetric
bilinear form on $M$, called the second fundamental form, $S$ is a tensor of type $(1,1)$, called the shape operator, and $\tau$ is a 1-form, called the transversal connection form.
\par We shall now consider the change of a transversal vector field for a given immersion $f$.

\begin{thm}[\cite{NS}]\label{tw::ChangeOfTransversalField}
Suppose we change a transversal vector field $C$ to
$$
\overline{C}=\Phi C+f_\ast(Z),
$$
where $Z$ is a tangent vector field on $M$ and $\Phi$ is a nowhere vanishing function on $M$. Then the affine fundamental form,
the induced connection, the transversal connection form, and the affine shape operator change as follows:
\begin{align*}
& \overline{h}=\frac{1}{\Phi}h,\\
& \overline{\nabla}_XY=\nabla_XY-\frac{1}{\Phi}h(X,Y)Z,\\
& \overline{\tau}=\tau+\frac{1}{\Phi}h(Z,\cdot)+d\ln|\Phi|,\\
& \overline{S}=\Phi S-\nabla_{\cdot}Z+\overline{\tau}(\cdot)Z.
\end{align*}
\end{thm}
\par If $h$ is nondegenerate, then we say that the hypersurface or the
hypersurface immersion is \emph{nondegenerate}. We have the following
\begin{thm}[\cite{NS}, Fundamental equations]\label{tw::FundamentalEquations}
For an arbitrary transversal vector field $C$ the induced
connection $\nabla$, the second fundamental form $h$, the shape
operator $S$, and the 1-form $\tau$ satisfy
the following equations:
\begin{align}
\label{eq::Gauss}&R(X,Y)Z=h(Y,Z)SX-h(X,Z)SY,\\
\label{eq::Codazzih}&(\nabla_X h)(Y,Z)+\tau(X)h(Y,Z)=(\nabla_Y h)(X,Z)+\tau(Y)h(X,Z),\\
\label{eq::CodazziS}&(\nabla_X S)(Y)-\tau(X)SY=(\nabla_Y S)(X)-\tau(Y)SX,\\
\label{eq::Ricci}&h(X,SY)-h(SX,Y)=2d\tau(X,Y).
\end{align}
\end{thm}
The equations (\ref{eq::Gauss}), (\ref{eq::Codazzih}),
(\ref{eq::CodazziS}), and (\ref{eq::Ricci}) are called the
equation of Gauss, Codazzi for $h$, Codazzi for $S$ and Ricci,
respectively.
\par For a hypersurface immersion $f\colon M\rightarrow \R^{n+1}$
a transversal vector field $C$ is said to be \emph{equiaffine}
(resp. \emph{locally equiaffine}) if $\tau=0$ (resp. $d\tau=0$).

\section{Almost paracontact structures}
\par A $(2n+1)$-dimensional manifold $M$ is said to have an
\emph{almost paracontact structure} if there exist on $M$ a tensor
field $\varphi$ of type (1,1), a vector field $\xi$ and a 1-form
$\eta$ which satisfy
\begin{align}
\varphi^2(X)&=X-\eta(X)\xi,\\
\eta(\xi)&=1
\end{align}
for every $X\in TM$
and the tensor field $\varphi$ induces an almost para-complex structure on the distribution $\DD=\operatorname{ker}\eta$. That is the eigendistributions $\DD ^{+},\DD ^{-}$ corresponding to the eigenvalues $1,-1$
of $\varphi$ have equal dimension $n$. Let $\nabla$ be a connection on $M$. We say that an almost paracontact structure $\FiXiEta$ is \emph{$\nabla$-parallel}
if $\nabla\varphi=0$, $\nabla\eta=0$ and $\nabla\xi=0$.
\par
By $\Ct$ we denote the real algebra of para-complex numbers, then the free $\Ct$-module $\Ct^{n+1}$ is a para-complex vector space.
We always assume that $\R^{2n+2}\cong\widetilde{\mathbb{C}}^{n+1}$ is endowed with
the standard para-complex structure $\widetilde{J}$ given by:
$$
\widetilde{J}(x_1,\ldots,x_{n+1},y_1,\ldots,y_{n+1})=(y_1,\ldots,y_{n+1},x_1,\ldots,x_{n+1}).
$$
For more details on para-complex geometry we refer to \cite{CFG,LS}.
\par Let $\dim M=2n+1$ and $f\colon M\rightarrow \R^{2n+2}$ be an affine hypersurface. Let $C$ be a transversal vector field on $M$. We say that $C$ is
\emph{$\widetilde{J}$-tangent} if $\widetilde{J}C_x\in f_\ast(T_xM)$ for every $x\in M$.
We also define a distribution $\DD$ on $M$ as the biggest $\widetilde{J}$-invariant distribution on $M$, that is
$$
\DD_x:=f_\ast^{-1}(f_\ast(T_xM)\cap \widetilde{J}(f_\ast(T_xM)))
$$
for every $x\in M$. We have that $\dim\DD _x\geq 2n$. If for some $x$ the $\dim\DD _x=2n+1$ then $\DD _x=T_xM$ and it is not possible to find a $\widetilde{J}$-tangent transversal vector field in a neighbourhood of $x$. Since we only study hypersurfaces with a $\widetilde{J}$-tangent transversal vector field, then we always have $\dim\DD=2n$. The distribution $\DD$ is smooth as an intersection of two smooth distributions and because $\dim \DD$ is constant.  A vector field $X$ is called a \emph{$\DD$-field}
if $X_x\in\DD_x$ for every $x\in M$. We use the notation $X\in\DD$ for vectors as well as for $\DD$-fields.

\par To simplify the writing, we will be omitting $f_\ast$ in front of vector fields in most cases.
\par Let $f\colon M\rightarrow \R^{2n+2}$ be an affine
hypersurface with a $\widetilde{J}$-tangent transversal vector field $C$. Then
we can define a vector field $\xi$, a 1-form $\eta$ and a tensor field
$\varphi$ of type (1,1) as follows:
\begin{align}
&\xi:=\widetilde{J}C;\\
\label{etanaD::eq::0}&\eta|_\DD=0 \text{ and } \eta(\xi)=1; \\
&\varphi|_\DD=\widetilde{J}|_\DD \text{ and } \varphi(\xi)=0.
\end{align}
It is easy to see that $(\varphi,\xi,\eta)$ is an almost paracontact
structure on $M$. This structure is called the \emph{induced almost
paracontact structure}.
Using Theorem \ref{tw::ChangeOfTransversalField} one may prove the following:
\begin{lem}[\cite{S1}]\label{lm::ZmianaStrukturyPrawieparakontaktowej}
Let $C$ be a $\widetilde{J}$-tangent transversal vector field. Then any other $\widetilde{J}$-tangent transversal vector field $\overline{C}$ has a form:
$$
\overline{C}=\Phi C+f_\ast Z,
$$
where $\Phi\neq 0$ and $Z\in\DD$. Moreover, if
$(\varphi,\xi,\eta)$ is an almost paracontact structure induced by $C$, then $\overline{C}$ induces an almost paracontact structure
$(\overline\varphi,\overline\xi,\overline\eta)$, where
$$
\begin{cases}
\overline\xi=\Phi\xi+\varphi Z,\\
\overline\eta=\frac{1}{\Phi}\eta,\\
\overline\varphi=\varphi-\eta(\cdot)\frac{1}{\Phi}Z.
\end{cases}
$$
\end{lem}

For an induced almost paracontact structure we have the following theorem
\begin{thm}[\cite{S3}]\label{tw::Wzory}
Let $f\colon M\rightarrow \mathbb{R}^{2n+2}$ be an affine hypersurface with a $\widetilde{J}$-tangent transversal vector field $C$.
If $(\varphi,\xi,\eta)$ is an induced almost paracontact structure on $M$
then the following equations hold:
\begin{align}
\label{Wzory::eq::1}&\eta(\nabla_XY)=h(X,\varphi Y)+X(\eta(Y))+\eta(Y)\tau(X),\\
\label{Wzory::eq::2}&\varphi(\nabla_XY)=\nabla_X\varphi Y-\eta(Y)SX-h(X,Y)\xi,\\
\label{Wzory::eq::3}&\eta([X,Y])=h(X,\varphi Y)-h(Y,\varphi X)+X(\eta(Y))-Y(\eta(X))\\
\nonumber &\qquad\qquad\quad+\eta(Y)\tau(X)-\eta(X)\tau(Y),\\
\label{Wzory::eq::4}&\varphi([X,Y])=\nabla_X\varphi Y-\nabla_Y\varphi X+\eta(X)SY-\eta(Y)SX,\\
\label{Wzory::eq::5}&\eta(\nabla_X\xi)=\tau(X),\\
\label{Wzory::eq::6}&\eta(SX)=-h(X,\xi)
\end{align}
for every $X,Y\in \X(M)$.
\end{thm}

\section{Parallel induced almost paracontact structures}
In this section we always assume that $(\varphi,\xi,\eta)$ is an almost paracontact structure induced by a $\Jt$-tangent
transversal vector field $C$ and $\nabla, h, S, \tau$ are affine objects induced by $C$ as well. Sometimes we denote a transversal vector field by $\overline{C}$ or even $\overline{\overline{C}}$. In such cases we use adequate (i. e. with bars) notation of induced objects.
\par When $\varphi$ is $\nabla$-parallel we have the following
\begin{lem}\label{lm::NablaPhiRowneZero}
Let  $(\FI, \xi ,\eta)$ be an induced almost paracontact structure such that $\nabla\FI=0$. Then
\begin{align}
\label{eq::fi1}&h|_{\DD\times\DD}=0,\\
\label{eq::fi2}&h(\xi,X)=h(X,\xi)=0\qquad\text{ for all $X\in\DD$},\\
\label{eq::fi3}&S|_\DD=0,\\
\label{eq::fi4}&S\xi=-h(\xi,\xi)\xi,\\
\label{eq::fi5}&d\tau=0,\\
\label{eq::fi6}&R=0.
\end{align}
\end{lem}
\begin{proof}
From the formula (\ref{Wzory::eq::2}) we have
$$
(\nabla_X\varphi)(Y)=\eta(Y)SX+h(X,Y)\xi
$$
for all $X,Y\in\XM$. Since $\nabla\varphi=0$
we get $h(X,Y)=0$ and $h(\xi,Y)=0$ for every $X,Y\in\DD$. Now, taking
 $X\in\DD$ and~$Y=\xi$ we have $SX=0$. Taking $X=Y=\xi$
we easily get $S\xi=-h(\xi,\xi)\xi$. The equation (\ref{eq::fi5}) follows immediately from the Ricci equation (\ref{eq::Ricci}). The last equation is an immediate consequence of the Gauss equation and (\ref{eq::fi1})--(\ref{eq::fi3}).
\end{proof}
By Lemma \ref{lm::NablaPhiRowneZero} we have that the transversal vector field $C$ is locally equiaffine, that is there exists (at least locally) a non-vanishing
function $\Phi$, such that $\overline{C}=\Phi C$ is equiaffine. Of course $\overline{C}$ is $\Jt$-tangent.
Now, using Theorem \ref{tw::ChangeOfTransversalField} and Lemma
\ref{lm::ZmianaStrukturyPrawieparakontaktowej} we get the following corollary
\begin{cor}\label{wn::ONablaPhi}
 Let $C$ be a $\Jt$-tangent transversal vector field such that
$\nabla\FI=0$ and let $\Phi$ be a nowhere vanishing function on $M$. Let us denote by $\overline C$ the transversal vector field $\Phi C$, then $\overline\nabla\overline\FI=0$ (actually $\overline\nabla=\nabla$ and $\overline\varphi=\varphi$). It means that the condition $\nabla\varphi=0$ is the direction property. In particular, locally we can always choose $C$ equiaffine.
\end{cor}
We shall prove

\begin{lem}\label{lm::NablaEtaRowneZero}
Let  $(\FI, \xi ,\eta)$ be an induced almost paracontact structure such that $\nabla \eta =0$. Then
\begin{align}
\label{eq::eta1} &h|_{\DD \times \DD}=0, &\\
\label{eq::eta2} &h(\xi,X)=h(X,\xi)=0 &\qquad\text{for every
$X\in
\DD $},\\
\label{eq::eta3} &\tau =0,&\\
\label{eq::eta4} &\nabla _{X}Y\in \DD &\qquad\text{for every
$X,Y\in
\DD $},\\
\label{eq::eta5} &\nabla _{X}\xi\in\DD &\qquad\text{for every
$X\in
\XM $},\\
\label{eq::eta6} &\nabla _{\xi}X\in\DD &\qquad\text{for every
$X\in
\DD $},\\
\label{eq::eta9} &X(h(\xi,\xi))=0 &\qquad\text{for every $X\in
\DD $}.
\end{align}
\end{lem}
\begin{proof}
Since $\nabla\eta=0$ we have
\begin{equation}\label{eq::NabEtaZero1}
\eta(\nabla_XY)=X(\eta(Y))
\end{equation}
for every $X,Y\in \XM$. Now, using the formula (\ref{Wzory::eq::1}) we get
\begin{equation}\label{eq::NabEtaZero2}
h(X,\FI Y)=-\eta(Y)\tau(X)
\end{equation}
for every $X,Y\in \XM$. Hence, if $X,Y\in\DD$, then $h(X,\FI Y)=0$, what proves (\ref{eq::eta1}).
Taking $X=\xi$ and $Y\in\DD$ in (\ref{eq::NabEtaZero2}) we easily get (\ref{eq::eta2}). On the other hand taking
 $Y=\xi$ we have $\tau(X)=0$, that is (\ref{eq::eta3}). The formulas (\ref{eq::eta4})--(\ref{eq::eta6})
can be directly obtained from (\ref{eq::NabEtaZero1}). To prove (\ref{eq::eta9}) let us note that from the Codazzi equation for $h$ (and using (\ref{eq::eta3})) we have
\begin{align*}
(\nabla_Xh)(\xi,\xi)=(\nabla_\xi h)(X,\xi)=\xi(h(X,\xi))-h(\nabla_\xi X,\xi)-h(X,\nabla_\xi\xi).
\end{align*}
Now, if we take $X\in\DD$ and use (\ref{eq::eta1})--(\ref{eq::eta2}) we get
$h(X,\xi)=0$ and $h(X,\nabla_\xi\xi)=0$, whereas (\ref{eq::eta6}) implies that we also have $h(\nabla_\xi X,\xi)=0$.
Thus, we obtain
$$
0=(\nabla_Xh)(\xi,\xi)=X(h(\xi,\xi))-2h(\nabla_X\xi,\xi)
$$
for every $X\in\DD$. Now, applying (\ref{eq::eta5}) in the above formula we get
$$
X(h(\xi,\xi))=0
$$
for every $X\in\DD$. This finishes the proof of (\ref{eq::eta9}).
\end{proof}

\begin{rem}\label{rem::1}
If $\nabla \varphi =0$ then $(\nabla _X\eta)Y=-\eta (Y)\tau (X)$.
\end{rem}
\begin{proof}
If $\nabla \varphi =0$ from (\ref{Wzory::eq::2}) we get that $h(X,\varphi Y)=0$ for all $X,Y\in \XM$. So from (\ref{Wzory::eq::1}) we obtain that $(\nabla _X\eta)Y=-\eta (Y)\tau (X)$.
\end{proof}
\begin{rem}\label{rem::2}
If $\nabla \varphi =0$ or $\nabla \eta =0$ then:
\begin{align}
\label{eq::rankh} \operatorname{rank}h\leq 1\\
\label{DDplusDminusnablaparallel} \DD, \DD^+, \DD^- \quad \text{are $\nabla$-parallel}\\
\label{DDplusDminusinvolutive} \DD, \DD^+, \DD^- \quad \text{are involutive}.
\end{align}
\end{rem}
\begin{proof}
The property (\ref{eq::rankh}) we immediately get from Lemma \ref{lm::NablaPhiRowneZero} or Lemma \ref{lm::NablaEtaRowneZero}.
By (\ref{Wzory::eq::1}) we have
\begin{align*}
\eta(\nabla_XY)=h(X,\varphi Y)
\end{align*}
for every $X\in\XM$ and $Y\in\DD$. Now (\ref{eq::fi1})--(\ref{eq::fi2}) (or (\ref{eq::eta1})--(\ref{eq::eta2})) imply that $\eta(\nabla_XY)=0$ for $X\in\XM$ and $Y\in\DD$, that is $\DD$ is $\nabla$-parallel. Using formula (\ref{Wzory::eq::2}) we get
\begin{align*}
\varphi(\nabla_XY)=\nabla_X\varphi Y=\nabla_XY
\end{align*}
for $X\in\XM$ and $Y\in\DD^+$, so $\nabla_XY\in\DD^+$. Similarly we obtain that $\nabla_XY\in\DD^-$ for $X\in\XM$ and $Y\in\DD^-$, that is both $\DD^+$ and $\DD^-$ are $\nabla$-parallel. Now (\ref{DDplusDminusinvolutive}) easily follows from (\ref{DDplusDminusnablaparallel}) and the fact that $\nabla$ is torsion-free.
\end{proof}

\begin{exa}
Let $f\colon \R^3\rightarrow \R^4$ be given by the formula
$$
f(x,y,z):=
\left[
\begin{matrix}
x+y\\
\sinh z\\
x-y\\
\cosh z
\end{matrix}
\right].
$$
Let $\{\partial_{x},\partial_{y},\partial_z\}$ be the canonical basis on $\R^3$ generated by the coordinate system $(x,y,z)$.
It easily follows that $f$ is an immersion and $\Jt f_x=f_x$, $\Jt f_y=-f_y$.
Moreover it is easy to see that
$$
C\colon \R^3\ni (x,y,z)\mapsto
\left[
\begin{matrix}
x\\
\sinh z\\
x\\
\cosh z
\end{matrix}
\right]
\in\R^4
$$
is a transversal $\Jt$-tangent vector field, since $C=\Jt f_z+xf_x$.
 Moreover, we have the following equalities
$$
\tau=0,\qquad S(\partial_x) =-\partial_x,\qquad S(\partial_y)=0,\qquad S(\partial_z)=-\partial_z
$$
and
$$
h=\left[
\begin{matrix}
0 & 0 & 0 \\
0 & 0 & 0 \\
0 & 0 & 1
\end{matrix}
\right].
$$
Now using the formula (\ref{Wzory::eq::1}) and the fact that $\partial_x\in\DD^+$, $\partial_y\in\DD^-$ and $\xi=\partial_z+x\partial_x$ we obtain that $\nabla\eta=0$, since $\varphi (\partial_z)=-x\partial_x$. However $\nabla \varphi \neq 0$. Indeed,
from (\ref{Wzory::eq::2}) we get
\begin{align*}
(\nabla _{\partial_z}\varphi )\partial_z&=\eta (\partial_z)S\partial_z+h(\partial_z,\partial_z)\xi
=-\partial_z+\xi=x\partial_x
\end{align*}
We also have
$$
R(\partial_x,\partial_z)\partial_z=h(\partial_z,\partial_z)S\partial_x-h(\partial_x,\partial_z)S\partial_z=-\partial_x,
$$
so $\nabla$ is not flat.
Now, let us consider $f$ with the transversal vector field
$$
\overline C\colon \R^3\ni (x,y,z)\mapsto
\left[
\begin{matrix}
0\\
\sinh z\\
0\\
\cosh z
\end{matrix}
\right]
\in\R^4.
$$
 Of course $\overline{C}$ is $\widetilde{J}$-tangent. In a similar way as above we compute that induced by $\overline C$ the almost paracontact structure $(\overline\varphi,\overline\xi,\overline\eta)$ is $\overline\nabla$-parallel. In particular $\overline\nabla$ is flat.\\
 Finally let us define
$$
N\colon \R^3\ni (x,y,z)\mapsto
\left[
\begin{matrix}
0\\
\frac{\sinh z}{\sqrt{\cosh^2z+\sinh^2z}}\\
0\\
-\frac{\cosh z}{\sqrt{\cosh^2z+\sinh^2z}}
\end{matrix}
\right]
\in\R^4.
$$
It is not difficult to see that $N$ is the normal field (in classical Riemannian sense) for $f$, since is orthogonal to $f_x$, $f_y$ and $f_z$ and $<N,N>=1$. Moreover $\Jt N$ is not tangent (if only $z\neq 0$), that is $N$ is not $\Jt$-tangent.
\end{exa}
The above example shows that in general condition $\nabla \eta=0$ is weaker than $\nabla \varphi =0$, however we shall show that if $\nabla \eta=0$ for some $\Jt$-tangent transversal vector field $C$ we can (at least locally) find another equiaffine $\Jt$-tangent transversal vector field $\overline{C}$ such that whole structure $(\overline{\varphi},\overline{\xi},\overline{\eta})$ is $\overline{\nabla}$-parallel. In order to prove it we will need the following lemmas

\begin{lem}\label{lm::FieldW}
If $\nabla\eta=0$ then there exists a vector field $W\in\X(M)$ such that the connection $\overline\nabla$ defined by
$$
\overline\nabla_XY:=\nabla_XY+h(X,Y)W
$$
is flat.
\end{lem}
\begin{proof}
From Theorem \ref{tw::ChangeOfTransversalField} there exist a non-vanishing function $\Phi$ and a vector field $Z_0\in\X(M)$ such that the normal vector field $\overline C$ is given by
$$
\overline C:=\Phi C+f_\ast Z_0.
$$
Let $\overline\nabla$, $\overline h$, $\overline S$ and $\overline \tau\equiv 0$ be affine objects induced by $\overline C$ and let $g$ be the first fundamental form on $M$ (i.e. Riemannian metric on $M$ induced from the canonical inner product $<\cdot,\cdot>$ on $\R^{2n+2}$).
We have
$$
g(\overline SX,Y)=\overline h(X,Y)=\frac{1}{\Phi}h(X,Y)
$$
for all $X,Y\in\XM$. Since $h(X,Y)=0$ for all $X\in\DD$ and $Y\in\XM$ and $g$ is the Riemannian metric on $M$ the above implies that $\overline S=0$ on $\DD$.
Now the Gauss equation
$$
\overline R(X,Y)Z=\overline h(Y,Z)\overline SX -\overline h(X,Z)\overline SY
$$
implies that $\overline R=0$, that is $\overline \nabla$ is flat. Since $\nabla$ and $\overline\nabla$ are related (see Th. \ref{tw::ChangeOfTransversalField}) by
$$
\overline\nabla_XY=\nabla_XY-\frac{1}{\Phi}h(X,Y)Z_0
$$
it is enough to take $W:=-\frac{1}{\Phi}Z_0$.
\end{proof}

\begin{lem}\label{lm::mapL1}
Let $f\colon M\rightarrow \mathbb{R}^{2n+2}$ be an affine hypersurface with a $\widetilde{J}$-tangent transversal vector field $C$ and
$(\varphi,\xi,\eta)$ be an induced almost paracontact structure on $M$. If $\nabla \eta =0$ then for every $p\in M$ there exist a neighbourhood $U$ of $p$ and a local basis $\{\partial_{1},\ldots ,\partial_{{2n+1}}\}$ on $U$  such that $\{\partial_{1},\ldots ,\partial_{{n}}\}$ span the distribution $\DD^+$,
$\{\partial_{{n+1}},\ldots ,\partial_{{2n}}\}$ span the distribution $\DD^-$ and $\nabla _{\partial_{i}}{\partial _{j}}=0$ for $i,j=1\ldots 2n$,
$\nabla _{\partial_{i}}{\partial _{2n+1}}=\nabla _{\partial_{2n+1}}{\partial _{i}}=0$ for $i=1,\ldots ,2n$.
\end{lem}
\begin{proof}
Let $p\in M$ and let $\overline\nabla$ be the connection from Lemma \ref{lm::FieldW}. Since $\overline\nabla$ is flat, in some neighborhood $U$ of $p$ there exist a basis
$\partial_1,\ldots,\partial_{2n+1}$ such that $\overline\nabla_{\partial_i}\partial_j=0$ for $i,j=1,\ldots,2n+1$. In particular we have $\overline\nabla_X\partial_i=0$
for $i=1,\ldots,2n+1$ and $X\in\X(U)$. Without loss of generality (shrinking the neighborhood $U$ if needed) we may
assume that $\partial_{2n+1}\notin\DD$. Then for $i=1,\ldots,2n$ we have the decomposition
$$
\partial_i=\partial_i^++\partial_i^-+\alpha_i\partial_{2n+1}
$$
where $\partial_i^+\in\DD^+$, $\partial_i^-\in\DD^-$ and $\alpha_i$ are some smooth functions on $U$.
Now for any $X\in \XU$ we have
\begin{align*}
0&=\overline\nabla_X{\partial_i}=\overline\nabla_X{\partial_i^+}+\overline\nabla_X{\partial_i^-}+\overline\nabla_X(\alpha_i\partial_{2n+1})\\
&=\overline\nabla_X{\partial_i^+}+\overline\nabla_X{\partial_i^-}+X(\alpha_i)\partial_{2n+1}\\
&=\nabla_X{\partial_i^+}+\nabla_X{\partial_i^-}+X(\alpha_i)\partial_{2n+1},
\end{align*}
where the last equality is an immediate consequence of the fact that $\overline\nabla_XY=\nabla_XY$ for all $X\in\XM$, $Y\in\DD$.
Since $\DD^+$ and $\DD^-$ are $\nabla$-parallel we obtain that
$\nabla_X{\partial_i^+}=0$ and $\nabla_X{\partial_i^-}=0$ for any $X\in\XU$ and $\alpha_i=\const$ for $i=1,\ldots,2n$.
In particular we have
\begin{align*}
\nabla _{\partial_i^+}{\partial_j^+}=0,\quad \nabla _{\partial_i^-}{\partial_j^-}=0,\quad \nabla _{\partial_i^+}{\partial_j^-}=0,\quad
\nabla _{\partial_i^-}{\partial_j^+}=0
\end{align*}
for $i,j=1,\ldots,2n$.
Let us consider
$$
\Span\{\partial_1^+,\ldots,\partial_{2n}^+\}\subset \DD^+ \quad\text{and}\quad \Span\{\partial_1^-,\ldots,\partial_{2n}^-\}\subset\DD^-.
$$
Since every $\partial_i$ is a linear combination of elements from $\{\partial_1^+,\ldots,\partial_{2n}^+\}$, $\{\partial_1^-,\ldots,\partial_{2n}^-\}$
and $\partial_{2n+1}$ we have
$$
\Span\{\partial_1^+,\ldots,\partial_{2n}^+\}\oplus \Span\{\partial_1^-,\ldots,\partial_{2n}^-\}\oplus \Span\{\partial_{2n+1}\}=TM.
$$
In particular
$\dim \Span\{\partial_1^+,\ldots,\partial_{2n}^+\}=\dim \Span\{\partial_1^-,\ldots,\partial_{2n}^-\}=n$.

Now (around $p$) we can choose $2n$ linearly independent vector fields $\{\partial_1',\ldots,\partial_{2n}'\}$ such that
$\partial_i'\in \{\partial_1^+,\ldots,\partial_{2n}^+\}$ and $\partial_{i+n}'\in \{\partial_1^-,\ldots,\partial_{2n}^-\}$ for $i=1,\ldots,n$.
We have
$$
\nabla_{\partial_i'}{\partial_j'}=0
$$
for $i,j=1,\ldots,2n$.
Note also that $\nabla_{\partial_i'}\partial_{2n+1}=\bar\nabla_{\partial_i'}\partial_{2n+1}=0$ and
$\nabla_{\partial_{2n+1}}{\partial_i'}=0$ for $i=1,\ldots,2n$.
Finally we see that the basis $\{\partial_1',\ldots,\partial_{2n}',\partial_{2n+1}\}$ has required properties.
\end{proof}

\begin{lem}\label{lm::mapL2}
Let $f\colon M\rightarrow \mathbb{R}^{2n+2}$ be an affine hypersurface with a $\widetilde{J}$-tangent transversal vector field $C$ and let
$(\varphi,\xi,\eta)$ be an induced almost paracontact structure on $M$. If $\nabla \eta =0$ then for every $p\in M$ there exist a neighbourhood $U$ of $p$
and a $\Jt$-tangent transversal vector field $\overline{C}$ defined on $U$ and a local coordinate system on $U$ $(x_1,\ldots ,x_{2n},y)$
such that $\partial_{x_1},\ldots,\partial_{x_n}\in \DD^+, \partial_{x_{n+1}},\ldots,\partial_{x_{2n}}\in \DD^-$ and the following conditions are satisfied:
\begin{align}
\label{eq::bar1}&\overline{\xi}=\partial _{y}; &\\
\label{eq::bar2}&\overline{\nabla}\overline{\eta}=0; &\\
\label{eq::bar3}&\overline{\nabla}_{\partial_{x_i}}\partial_{x_j}=0 &\text{for }i,j=1,\ldots,2n;\\
\label{eq::bar4}&\overline{\nabla}_{\partial_{x_i}}\partial_{y}=\overline{\nabla}_{\partial_{y}}\partial_{x_i}=0 &\text{for }i=1,\ldots,2n.
\end{align}
where $(\overline{\varphi},\overline{\xi},\overline{\eta})$ and $\overline{\nabla}$ are the almost paracontact structure and the affine connection induced by $\overline{C}$, respectively.
\end{lem}

\begin{proof}
By Lemma \ref{lm::mapL1} for any $p\in M$ there exist a neighbourhood $U$ of $p$ and a local basis $\{\partial_{1},\ldots ,\partial_{{2n+1}}\}$ on $U$
with properties described in the Lemma \ref{lm::mapL1}.
Of course $\eta (\partial_{2n+1})\neq 0$.
Since $\nabla _{\partial_{i}}{\partial _{2n+1}}=0$ for $i=1,\ldots,2n$, using (\ref{Wzory::eq::1}) we obtain
\begin{align}\label{eq::etaNablaxiy}
0=\eta (\nabla _{\partial_{i}}{\partial _{2n+1}})=\partial_{i}(\eta (\partial_{2n+1}))
\end{align} for $i=1,\ldots,2n$.
Let us define
$$
Y:=\frac{1}{\eta (\partial_{2n+1})}\cdot\partial_{2n+1}.
$$
Thanks to (\ref{eq::etaNablaxiy}) we get
\begin{align*}
\nabla _{\partial_{i}}Y=\nabla_Y\partial_{i}=0
\end{align*} for $i=1,\ldots,2n$.
In particular $[\partial_{i},Y]=0$ for $i=1,\ldots,2n$ and there exist a local coordinate system $(x_1,\ldots,x_{2n},y)$ around $p$  such that
$\partial_i=\partial_{x_{i}}$ for $i=1,\ldots,2n$ and $Y=\partial_{y}$. We also have $\partial_{x_{i}}\in \DD^+$ and $\partial_{x_{n+i}}\in \DD^-$ for $i=1,\ldots,n$.
Now, we have
\begin{align*}
\eta (\partial_{y})=\eta (Y)=1=\eta (\xi)
\end{align*}
that is there exists $Z\in \DD$ such that $\partial_{y}=\xi +Z$.
Let us define
\begin{align*}
\overline{C}:=C+f_{\ast}(\varphi Z).
\end{align*}
By Theorem \ref{tw::ChangeOfTransversalField} and Lemma \ref{lm::ZmianaStrukturyPrawieparakontaktowej} we have
\begin{align*}
\overline{\xi}=\xi+Z=\partial_{y}, \quad \overline{\eta}=\eta
\end{align*}
and
\begin{align*}
\overline{\nabla}_XY=\nabla_XY-h(X,Y)\varphi Z \quad \text{for} \quad X,Y\in\XM.
\end{align*}
In consequence
\begin{align*}
\overline{\nabla}_{\partial_{x_i}}\partial_{x_j}=\nabla_{\partial_{x_i}}\partial_{x_j}=0
\end{align*}
and
\begin{align*}
\overline{\nabla}_{\partial_{x_i}}\partial_{y}=\overline{\nabla}_{\partial_{y}}\partial_{x_i}=\nabla_{\partial_{y}}\partial_{x_i}=\nabla_{\partial_{x_i}}\partial_{y}=0
\end{align*}
for $i,j=1,\ldots,2n$.
Finally we obtain
\begin{align*}
(\overline{\nabla}_X\overline{\eta})(Y)&=X(\overline{\eta}(Y))-\overline{\eta}(\overline{\nabla}_XY)=X(\eta (Y))-\eta (\nabla_XY-h(X,Y)\varphi Z)\\
&=(\nabla_X\eta)(Y)+h(X,Y)\eta(\varphi Z)=(\nabla_X\eta)(Y)=0.
\end{align*}
\end{proof}

Now we can prove the following

\begin{thm}\label{tw::mapL3}
Let $f\colon M\rightarrow \mathbb{R}^{2n+2}$ be an affine hypersurface with a $\widetilde{J}$-tangent transversal vector field $C$ and
$(\varphi,\xi,\eta)$ be an induced almost paracontact structure on $M$. If $\nabla \eta =0$ then for every $p\in M$ there exist a neighbourhood $U$ of $p$ and a $\Jt$-tangent transversal vector field $\overline{\overline{C}}$ defined on $U$ such that the induced almost paracontact structure
$(\overline{\overline{\varphi}},\overline{\overline{\xi}},\overline{\overline{\eta}})$ is $\overline{\overline{\nabla}}$-parallel.
Moreover around $p$ there exists a local coordinate system $(x_1,\ldots,x_{2n},x_{2n+1})$ such that
$\partial_{x_1},\ldots,\partial_{x_n}\in \DD^+, \partial_{x_{n+1}},\ldots,\partial_{x_{2n}}\in \DD^-$, $\nabla_{\partial_{x_i}}\partial_{x_j}=0$ for $i,j=1,\ldots,2n+1$ and $\partial_{x_{2n+1}}=\overline{\overline{\xi}}$.
\end{thm}
\begin{proof}
By Lemma \ref{lm::mapL2} for any $p\in M$ there exist a neighbourhood $U$ of $p$, local coordinates $(x_1,\ldots ,x_{2n},y)$ on $U$ and  a $\Jt$-tangent transversal vector field $\overline{C}$ defined on $U$ such that (\ref{eq::bar1})--(\ref{eq::bar4}) are satisfied.
Using formula (\ref{Wzory::eq::2}), for $i=1,\ldots,2n$ we have
\begin{align*}
\overline{S}\partial_{x_i}=-\overline{\varphi}(\overline{\nabla}_{\partial_{x_i}}\overline{\xi})=-\overline{\varphi}(\overline{\nabla}_{\partial_{x_i}}\partial_y)=0,
\end{align*}
where $\overline{S}$ is the shape operator induced by $\overline{C}$. In particular $\overline{S}|\DD=0$.
From (\ref{Wzory::eq::1}) we get $\overline{\eta}(\overline{\nabla}_{\overline{\xi}}\overline{\xi})=0$ that is
$\overline{\nabla}_{\overline{\xi}}\overline{\xi}\in\DD$.
Therefore there exist smooth functions $p_i$ ($i=1,\ldots,2n$) such that
\begin{align*}
\overline{\nabla}_{\overline{\xi}}\overline{\xi}=\sum_{i=1}^{2n}p_i\partial_{x_i}.
\end{align*}
Now from the Gauss equation we obtain that for every $i=1,\ldots,2n$
\begin{align*}
0=\overline{R}(\partial_{x_i},\partial_y)\partial_y
=\overline{\nabla}_{\partial_{x_i}}\overline{\nabla}_{\partial_y}\partial_y
=\overline{\nabla}_{\partial_{x_i}}\sum_{j=1}^{2n}p_j\partial_{x_j}
=\sum_{j=1}^{2n}\partial_{x_i}(p_j)\partial_{x_j}
\end{align*}
therefore $p_j$ depends only on $y$ for $j=1,\ldots,2n$.
Let us define
\begin{align*}
Z:=\sum_{i=1}^{2n}a_i\partial_{x_i},
\end{align*}
where $a_i$ are smooth functions defined by
\begin{align}
\label{eq::ai} a_i&:=-e^{\int \overline{h}(\partial_y,\partial_y)dy}\cdot \int p_ie^{-\int \overline{h}(\partial_y,\partial_y)dy}dy\\
\label{eq::ain}a_{i+n}&:=-e^{-\int \overline{h}(\partial_y,\partial_y)dy}\cdot \int p_{i+n}e^{\int \overline{h}(\partial_y,\partial_y)dy}dy
\end{align} for $i=1,\ldots,n$.
By (\ref{eq::eta9}) we have $\partial_{x_i}(\overline{h}(\partial_y,\partial_y))=0$ for $i=1,\ldots,2n$, that is $\overline{h}(\partial_y,\partial_y)$
depends only on $y$ and in consequence $a_i$ depends only on $y$ for $i=1,\ldots,2n$.
Now we can define another transversal vector field:
\begin{align*}
\overline{\overline{C}}:=\overline{C}+f_{\ast}\varphi Z.
\end{align*}
Since $Z\in\DD$ we see that $\overline{\overline{C}}$ is $\Jt$-tangent.
First note that
\begin{align*}
\overline{\overline{\nabla}}_{\overline{\overline{\xi}}}\partial_{x_i}=0,\\
\overline{\overline{\nabla}}_{\partial_{x_i}}\overline{\overline{\xi}}=0,\\
\overline{\overline{\nabla}}_{\partial_{x_i}}\partial_{x_j}=0
\end{align*} for $i,j=1,\ldots,2n$.
Indeed, the above follows immediately from the fact that $\overline{\overline{\xi}}=\overline{\xi}+Z=\partial_y+Z$,
 $a_i$ depend only on $y$ and $\overline{\overline{\nabla}}_XY=\overline{\nabla}_XY$ if only $X\in\DD$ or $Y\in \DD$.
 Let us denote $a_i':=\frac{\partial a_i}{\partial y}$.
One may compute
\begin{align*}
\overline{\overline{\nabla}}_{\overline{\overline{\xi}}}\overline{\overline{\xi}}&=\overline{\nabla}_{\overline{\overline{\xi}}}\overline{\overline{\xi}}-\overline{h}(\overline{\overline{\xi}},\overline{\overline{\xi}})\varphi Z=\overline{\nabla}_{\overline{\xi}+Z}{(\overline{\xi}+Z)}-\overline{h}(\overline{\xi}+Z,\overline{\xi}+Z)\varphi Z\\&=\overline{\nabla}_{\overline{\xi}}\overline{\xi}+\overline{\nabla}_{\overline{\xi}}Z+\overline{\nabla}_{Z}\overline{\xi}+\overline{\nabla}_{Z}Z-\overline{h}(\overline{\xi},\overline{\xi})\varphi Z\\&=\overline{\nabla}_{\partial_y}\partial_y+\overline{\nabla}_{\partial_y}Z+\overline{\nabla}_ZZ-\overline{h}(\partial_y,\partial_y)\varphi Z\\
&=\sum_{i=1}^{2n}p_i\partial_{x_i}+\overline{\nabla}_{\partial_y}(\sum_{i=1}^{2n}a_i\partial_{x_i})-\overline{h}(\partial_y,\partial_y)\varphi (\sum_{i=1}^{2n}a_i\partial_{x_i})\\
&=\sum_{i=1}^{2n}p_i\partial_{x_i}+\sum_{i=1}^{2n}a_i'\partial_{x_i}-\overline{h}(\partial_y,\partial_y)\sum_{i=1}^na_i\partial_{x_i}+\overline{h}(\partial_y,\partial_y)\sum_{i=1}^na_{i+n}\partial_{x_{i+n}}\\
&=\sum_{i=1}^{n}p_i\partial_{x_i}+\sum_{i=1}^{n}a_i'\partial_{x_i}-\overline{h}(\partial_y,\partial_y)\sum_{i=1}^na_i\partial_{x_i}\\
&+\sum_{i=1}^{n}p_{i+n}\partial_{x_{i+n}}+\sum_{i=1}^{n}a_{i+n}'\partial_{x_{i+n}}+\overline{h}(\partial_y,\partial_y)\sum_{i=1}^na_{i+n}\partial_{x_{i+n}}=0,
\end{align*}
where the last equality easily follows from (\ref{eq::ai}) and (\ref{eq::ain}).
Now, using the above we have
\begin{align*}
(\overline{\overline{\nabla}}_{\overline{\overline{\xi}}}\overline{\overline{\varphi}})\overline{\overline{\xi}}=-\overline{\overline{\varphi}}(\overline{\overline{\nabla}}_{\overline{\overline{\xi}}}\overline{\overline{\xi}})
=0,\\
(\overline{\overline{\nabla}}_{\overline{\overline{\xi}}}\overline{\overline{\varphi}})(\partial_{x_i})=\overline{\overline{\nabla}}_{\overline{\overline{\xi}}}(\overline{\overline{\varphi}}(\partial_{x_i}))
-\overline{\overline{\varphi}}(\overline{\overline{\nabla}}_{\overline{\overline{\xi}}}\partial_{x_i})=0,\\
(\overline{\overline{\nabla}}_{\partial_{x_i}}\overline{\overline{\varphi}})(\overline{\overline{\xi}})=-\overline{\overline{\varphi}}(\overline{\overline{\nabla}}_{\partial_{x_i}}\overline{\overline{\xi}})=0,\\
(\overline{\overline{\nabla}}_{\partial_{x_i}}\overline{\overline{\varphi}})(\partial_{x_j})=0,
\end{align*}
that is $\overline{\overline{\nabla}}\overline{\overline{\varphi}}=0$.
Since $\overline{\overline{\nabla}}_{\overline{\overline{\xi}}}\overline{\overline{\xi}}=0$ and $\overline{\overline{\nabla}}_{\partial_{x_i}}\overline{\overline{\xi}}=0$ for $i=1,\ldots,2n$ we get that
$\overline{\overline{\nabla}}_{X}\overline{\overline{\xi}}=0$ for $X\in\X(U)$, that is $\overline{\overline{\nabla}}\overline{\overline{\xi}}=0$.
Since $\overline\tau=0$ Theorem \ref{tw::ChangeOfTransversalField} implies that $\overline{\overline{\tau}}=0$ and now Remark \ref{rem::1}
implies that $\overline{\overline{\nabla}}\overline{\overline{\eta}}=0$.
\par Finally note that for $\partial_{x_1},\ldots,\partial_{x_{2n}},\partial_{x_{2n+1}}:=\overline{\overline{\xi}}$ we have $[\partial_{x_i},\partial_{x_j}]=0$
for $i,j=1,\ldots,2n+1$, so there exists a local coordinate system $(x_1,\ldots,x_{2n},x_{2n+1})$ with desired properties.
The proof is completed.
\end{proof}

Now we can state the following classification theorem. Namely we have

\begin{thm}\label{thm::opostacif}
Let $f\colon M\rightarrow \mathbb{R}^{2n+2}$ be an affine hypersurface with a $\widetilde{J}$-tangent transversal vector field $C$ and
$(\varphi,\xi,\eta)$ be an induced almost paracontact structure on $M$. If $\varphi$ or $\eta$ is $\nabla$-parallel then for every $p\in M$
there exists a neighbourhood $U$ of $p$ such that $f$ can be locally expressed in the form
\begin{align}\label{eq::WzorNafNew}
f(x_1,\ldots,x_{2n},y)=x_1b_1+\cdots+x_{2n}b_{2n}+\Jt v\int\cosh {\alpha}(y)\,dy\\\nonumber+v\int\sinh {\alpha}(y) \,dy,
\end{align}
where $b_1,\ldots,b_{2n}, v\in\R^{2n+2}$, $b_1,\ldots,b_{2n}, v, \Jt v$ are linearly independent and such that $\Jt b_i=b_i$ for $i=1,\ldots, n$,
$\Jt b_i=-b_i$ for $i=n+1,\ldots, 2n$
and $\alpha$ is some smooth function. Moreover, the converse is true in the sense that for the function {\rm{(\ref{eq::WzorNafNew})}}
there exists a global $\Jt$-tangent equiaffine transversal vector field $C$ such that $(\varphi,\xi,\eta)$ is $\nabla$-parallel.
\end{thm}
\begin{proof}
Let $p\in M$. Thanks to Theorem \ref{tw::mapL3}, Corollary  \ref{wn::ONablaPhi} and Remark \ref{rem::1}, without loss of generality,
we may assume that $(\varphi,\xi,\eta)$ is $\nabla$-parallel around $p$.
Let $(x_1,\ldots,x_{2n},x_{2n+1})$ be the local coordinate system from the Theorem \ref{tw::mapL3}.
Let us denote $x_{2n+1}=y$.
Now, by the Gauss formula we have the following system of differential equations:
\begin{align}
\label{eq::dffxixj}f_{x_ix_j}=0,\\
\label{eq::dffxiy}f_{x_iy}=0,\\
\label{eq::dffyy}f_{yy}=h(\xi,\xi)C=h(\xi,\xi)\Jt f_y
\end{align} for $i=1,\ldots,2n$.
Solving (\ref{eq::dffxixj}) and (\ref{eq::dffxiy}) we obtain
\begin{align*}
f(x_1,\ldots,x_{2n},y)=\sum_{i=1}^{2n}x_ib_i+A(y),
\end{align*}
where $b_1,\ldots,b_{2n}\in\R^{2n+2}$, $b_1,\ldots,b_{2n}$ are linearly independent and such that $\Jt b_i=b_i$ for $i=1,\ldots, n$,
$\Jt b_i=-b_i$ for $i=n+1,\ldots, 2n$ and $A$ is some smooth function depending only on variable $y$ with values in $\R^{2n+2}$.
Now (\ref{eq::dffyy}) takes the form
\begin{align}\label{eq::DEwithA}
A''=\beta\Jt A',
\end{align}
where $\beta:=h(\xi,\xi)$. Substituting $G:=A'$ we get
\begin{align}\label{eq::DEwithG}
G'=\beta\Jt G.
\end{align}
Let $\overline\beta$ be any integral of $\beta$.
First note that for every $v\in\R^{2n+2}$ the function
\begin{align}\label{eq::DEwithG-Solution}
G(y)= \Jt v\cosh \overline{\beta}(y)+v\sinh \overline{\beta}(y)
\end{align}
is the solution of (\ref{eq::DEwithG}). On the other hand, since (\ref{eq::DEwithG}) is a first order ordinary differential equation
all its solutions have the form (\ref{eq::DEwithG-Solution}).
The above imply that solutions of (\ref{eq::DEwithA}) have the form
$$
A(y)= \Jt v\int \cosh \overline{\beta}(y)dy+v\int \sinh \overline{\beta}(y)dy.
$$
Since $f$ is the immersion and $C$ is transversal we have that $b_1,\ldots,b_{2n}$, $f_y=G(y)$ and $C=\Jt f_y$ are linearly independent. In particular,
we obtain that $b_1,\ldots,b_{2n},v,\Jt v$ are linearly independent too. Indeed, if we assume that
$$
\sum_{i=1}^{2n}a_ib_i+a_{2n+1}v+a_{2n+2}\Jt v=0
$$
for some functions $a_1,\ldots,a_{2n+2}$ we get
\begin{align*}
0&=\sum_{i=1}^{2n}a_ib_i+a_{2n+1}v+a_{2n+2}\Jt v \\
&=\sum_{i=1}^{2n}a_ib_i+(a_{2n+2}\cosh \overline{\beta}(y)-a_{2n+1}\sinh \overline{\beta}(y))f_y \\
&\phantom{=}+(a_{2n+1}\cosh \overline{\beta}(y)-a_{2n+2}\sinh \overline{\beta}(y))\Jt f_y.
\end{align*}
Now, since $\{b_1,\ldots,b_{2n},f_y,\Jt f_y \}$ are linearly independent we obtain $a_1=\cdots =a_{2n}=0$ and
$$a_{2n+2}\cosh \overline{\beta}(y)-a_{2n+1}\sinh \overline{\beta}(y)=a_{2n+1}\cosh \overline{\beta}(y)-a_{2n+2}\sinh \overline{\beta}(y)=0.$$
The above implies that $a_{2n+1}=a_{2n+2}=0$.
Summarising $f$ can be locally expressed in the form:
\begin{align*}
f(x_1,\ldots,x_{2n},y)=x_1b_1+\cdots+x_{2n}b_{2n}+\Jt v\int\cosh \overline{\beta}(y)\,dy\\\nonumber+v\int\sinh \overline{\beta}(y) \,dy.
\end{align*}
Now denoting $\alpha:=\overline{\beta}$ we get the thesis.
\par In order to prove the last part of the theorem let us note that since $b_1,\ldots,b_{2n},v,\Jt v$ are linearly independent the function $f$
given by (\ref{eq::WzorNafNew}) is an immersion and $C:=\Jt f_y$ is transversal and $\Jt$-tangent. Let $(\varphi,\xi,\eta)$ be an almost paracontact structure induced by $C$.
Since $C_{x_i}=0$ and $C_{y}=\alpha'f_y$ we have $\tau=0$, $S|_\DD=0$ and $S\xi=-\alpha'\xi$.
Since $f_{x_ix_j}=0$, $f_{x_iy}=0$ and $f_{yy}=\alpha'C$ we have $h(X,Y)=0$ if only $X\in\DD$ or $Y\in\DD$ and $h(\xi,\xi)=\alpha'$.
Now using (\ref{Wzory::eq::2}) we easily obtain that $\nabla\varphi =0$. Since $C$ is equiaffine we also have $\nabla\eta=0$ and $\nabla\xi=0$.
\end{proof}

\begin{cor}\label{wn::hyperplane}
If $\nabla\varphi=0$ or $\nabla\eta=0$ and $\operatorname{rank}h=0$ on $M$ then $f$ is a piece of a hyperplane.
\end{cor}
\begin{proof}
By Theorem \ref{thm::opostacif} $f$ can be locally expressed in the form (\ref{eq::WzorNafNew}). In particular
\begin{align*}
f_{yy}=\alpha'(y) \cdot (\Jt v\sinh \alpha(y)+v\cosh \alpha (y))=\alpha'(y)\Jt f_y.
\end{align*}
On the other hand $f_{yy}$ is tangent, since $h=0$. Since $\Jt f_y$ is transversal we obtain $\alpha '=0$, that is $\alpha=\const$ and, in consequence, $f$ is a piece of a hyperplane.
\end{proof}

We conclude this section with an example showing that the condition $\nabla\xi=0$ is much more weaker than $\nabla\varphi=0$ or $\nabla\eta=0$.

\begin{exa}
Let us consider an affine immersion $f\colon (0,\infty)^2\times\R\rightarrow \R^4$ given by the formula
$$
f(x,y,z):=
\left[
\begin{matrix}
\frac{1}{2}(x^2+y^2)\\
\sinh z\\
\frac{1}{2}(x^2-y^2)\\
\frac{1}{3}(x^3+y^3)+\cosh z
\end{matrix}
\right].
$$
It is easy to verify that
$$
C\colon (0,\infty)^2\times\R\ni (x,y,z)\mapsto
\left[
\begin{matrix}
0\\
\sinh z\\
0\\
\cosh z
\end{matrix}
\right]
\in\R^4
$$
is a $\Jt$-tangent transversal vector field for $f$.
Let $\{\partial_{x},\partial_{y},\partial_z\}$ be the canonical basis on $(0,\infty)^2\times\R$ generated by the coordinate system $(x,y,z)$ and
let $(\varphi,\xi,\eta)$ be an almost paracontact structure induced by $C$. Note that we have $\xi=\partial_z$.
By straightforward computations we obtain
$$
h=\left[
\begin{matrix}
x\cosh z & 0 & 0 \\
0 & y\cosh z & 0 \\
0 & 0 & 1
\end{matrix}
\right],
$$
so in particular $f$ is nondegenerate, that is $\rank h=3$. Now by Remark \ref{rem::2} it is not possible to find a $\Jt$-tangent transversal vector field
such that $\nabla\varphi=0$ or $\nabla\eta=0$. However, by the Gauss formula we have
$$
\nabla_{\partial_x}\partial_z=\nabla_{\partial_y}\partial_z=\nabla_{\partial_z}\partial_z=0,
$$
that is $\nabla\xi=0$.
\end{exa}

\emph{This Research was financed by the Ministry of Science and Higher Education of the Republic of Poland.}

\end{document}